\documentclass[oneside,12pt, reqno]{amsart}
\usepackage{amscd}
\usepackage{amssymb}
\usepackage{verbatim}

\allowdisplaybreaks[1]

\usepackage[all]{xy}

\makeatletter
\def\section{\@startsection{section}{1}%
  \z@{.7\linespacing\@plus\linespacing}{.5\linespacing}%
  {\normalfont\centering}}
\makeatother

\newtheorem{thm}{Theorem}[section]

\newtheorem{lem}[thm]{Lemma}

\theoremstyle{definition}
\newtheorem{ex}[thm]{Example}

\numberwithin{equation}{section}

\newcommand{\be}{\begin{equation}}

\newcommand{\bmm}{\left[\begin{array}}
\newcommand{\emm}{\end{array}\right]}

\newcommand{\half}{\frac{1}{2}}

\begin{document}
\pagenumbering{arabic}

\title{Schur Multipliers of Nilpotent Lie Algebras}

\maketitle
\begin{center}{
{\footnotesize{LINDSEY R. BOSKO  \qquad lrbosko@ncsu.edu \\ ERNIE L. STITZINGER \qquad stitz@ncsu.edu} \\
\emph{Mathematics Department, North Carolina State University \\
Raleigh, North Carolina 27695, United States of America \\
}}}
\end{center}

\begin{abstract}
We consider the Schur multipliers of finite dimensional nilpotent Lie algebras.  If the algebra has dimension greater than one, then the Schur multiplier is non-zero.  We give a direct proof of an upper bound for the dimension of the Schur multiplier as a function of class and the minimum number of generators of the algebra.  We then compare this bound with another known bound. \\

\end{abstract}

\section{Introduction}
The Lie algebra analogue of the Schur multiplier was investigated in the dissertations of Kay Moneyhun and Peggy Batten (see \cite{Moneyhun} and \cite{Batten}).  Among their results is that if $\dim L=n$, then $\dim M(L) \leq \half n(n-1)$ where $M(L)$ is the Schur multiplier of the Lie algebra, $L$.  A number of other results bounding $\dim M(L)$ have appeared (see \cite{Bosko}, \cite{Hardy}, \cite{Niroomand}, \cite{Salemkar}, and \cite{Yankosky}).  In this note several other bounds are provided.  In particular, we show that if $L$ is nilpotent and $\dim L>1$, then $\dim M(L) \neq 0$.   This result can be contrasted with a result of Johnson in \cite{Johnson} which shows that a $p-$group with trivial multiplier has restrictions placed on it.  We also find an upper bound for $\dim M(L)$ as a function of class and the number of generators for $L$. This result is similar to Theorem 3.2.5 of \cite{Karpilovsky}.  This Lie algebra result follows as a consequence of a general theory, provided in \cite{Salemkar} but we give a short, direct proof and contrast this bound with one found in \cite{Hardy}.

\section{Preliminaries}
Suppose that $L$ is generated by $n$ elements.  Let $F$ be a free Lie algebra generated by $n$ elements and $L \cong F/R$.  Since $R$ is an ideal in $F$, $R$ is also free.  Witt's formula from \cite{Bahturin} gives us
	\begin{equation}
	\label{Witt}
	\dim F^d/F^{d+1} = \frac{1}{d} \sum_{m \mid d} \mu(m) n^{d/m} \equiv l_n(d)
	\end{equation}
where $\mu$ is the M\"obius function.  Hence, $F/F^t$ is finite dimensional and nilpotent for all $t$.  

Let $N$ be an ideal in $L$ and $S$ be an ideal in $F$ such that $(S+R)/R \cong N$.  Recall that $M(L)=(F^2 \cap R)/[F,R]$ (\cite{Batten}).  Then $M(L/N) \cong (F^2 \cap (S+R))/[F,S+R]$.  It is routine to verify that there is a natural exact sequence
	\begin{equation}
	\label{exact1}
	0 \rightarrow \frac{R \cap [F,S]}{[F,R] \cap [F,S]} \rightarrow M(L) \rightarrow M(L/N) \rightarrow \frac{N \cap L^2}{[N,L]} \rightarrow 0.
	\end{equation}

Note that covers and multipliers can be computed using the GAP program \cite{Ellis}.

\section{Trivial M(L)}
It is shown in \cite{Johnson} that if $G$ is a finite $p-$group with $M(G)=e$, then severe restrictions are placed on $G$.  For further work in the problem see \cite{Webb} and also \cite{Wiegold} for a simpler proof.  We will show that if $L$ is a nilpotent Lie algebra with $M(L)=0$ then $\dim L \leq 1$.

Let $L$ be a nilpotent Lie algebra generated by $n > 1$ elements.  Hence, $\dim L/L^2=n$.  Let $F$ be a free Lie algebra generated by $n$ elements with $L \cong F/R$.  Suppose that $L$ has class $c$.  Hence, $F^{c+2} \subsetneq F^{c+1} \subseteq R$ using the result in the last section.  Furthermore, $F/F^{c+2}$ is finite dimensional and nilpotent of class $c+1$.  Then, 
	$$n= \dim L/L^2 = \dim \frac{F/R}{(F/R)^2} = \dim \frac{F}{F^2+R} \leq \dim F/F^2 = n.$$

\begin{lem}
If $L$ is nilpotent and $\dim L=n > 1$, then $R \subseteq F^2$ and $M(L) \cong R/[F,R]$.
\end{lem}

\begin{thm}
If $L$ is a finite dimensional nilpotent Lie algebra of dimension greater than $1$ and class $c$, then $M(L) \neq 0$.
\end{thm}

\begin{proof}
Continuing with the notation, $L \cong F/R$, $F^{c+2} \subsetneq F^{c+1} \subseteq R$, and $F/F^{c+2}$ is nilpotent.  Hence, $[F,R] \subsetneq R$ and $M(L) \cong R/[F,R] \neq 0$.
\end{proof}

\section{An Upper Bound for $\dim M(L)$}

Suppose $L$ has class $c \geq 2$.  Let $N=L^c$ and $S=F^c$ in Eq. \ref{exact1}.  Then $L^c \cong (F^c+R)/R$ and $[F,S] = F^{c+1} \subseteq R$ since $L^{c+1}=0$.  Hence, Eq. \ref{exact1} becomes
	\begin{equation}
	\label{exact2}
	0 \rightarrow \frac{F^{c+1}}{[F,R] \cap F^{c+1}} \overset{\sigma}{\rightarrow} M(L) \rightarrow M(L/L^c) \rightarrow L^c \rightarrow 0.
	\end{equation}
	
\begin{thm}
\label{M(L)Bound1}
Let $L$ be a nilpotent Lie algebra of class $c$ which is generated by $n$ elements.  Then
	$$\dim M(L) \leq \sum_{j=1}^c l_n(j+1)$$
where $l_n (q) = \frac{1}{q} \sum_{s \mid q} \mu(s) n^{q/s}$.
\end{thm}

\begin{proof}
Induct on $c$.  Let $F$ be free of rank $n$ where $L \cong F/R$.  If $c=1$, then $F/R$ is abelian, $F^2 \subseteq R$ and $M(L) = F^2/[F,R]$.  Since $F^3 \subseteq [F,R], \ \dim M(L) \leq \dim F^2/F^3=l_n (2)$.  Now, suppose that $c>1$.  By induction, $\dim M(L/L^c) \leq t \equiv \sum_{j=1}^{c-1} l_n (j+1)$.  In Eq. \ref{exact2}, let $A= Im (\sigma)$.  Then $M(L)/A \cong B \subseteq M(L/L^c)$.  Hence, $\dim M(L)/A \leq t$.  But $F^{c+1} \subseteq R$ and $F^{c+2} \subseteq [F,R] \cap F^{c+1}$.  Thus, $A$ is the homomorphic image of $F^{c+1}/F^{c+2}$.  Therefore $\dim A \leq l_n(c+1)$ by Eq. \ref{Witt} and $\dim M(L) \leq t + l_n (c+1)$ as desired.
\end{proof}

We compare our result to the upper bound given in \cite{Hardy}:

\begin{thm}
\label{M(L)Bound2}
If $L$ is a Lie algebra of dimension $n$, then 
	$$\dim M(L) \leq \half n(n-1) - \dim L^2.$$
\end{thm}

We now examine the two theorems applied to different Lie algebras.
\begin{ex}
Let $F$ be a free Lie algebra on $2$ generators and $L=F/F^3$.  Then $L$ is a Lie algebra of $2$ generators and class $2$.  So, $L \supseteq L^2 \supseteq L^3 = 0$.  Then, $\dim L/L^2 = l_2(1)=2$, $\dim L^2/L^3 = l_2(2)=\half[\mu(1)2^2+\mu(2)2]=\half(4-2)=1$.  Thus, $\dim L = 3$ and by Theorem \ref{M(L)Bound2}, $\dim M(L) \leq 2$.  By Theorem \ref{M(L)Bound1}, 
	\begin{align*}
		\dim M(L) \leq \sum_{j=1}^2 l_2(j+1) &= l_2(2) + l_2(3) \\
		&=1 + \frac{1}{3}(\mu(1)2^3 + \mu(3)2) \\
		&=1 + \frac{1}{3}(6)=3.
	\end{align*}
Thus, the result of Theorem \ref{M(L)Bound2} proves to be a better bound than the one obtained by our new theorem.
\end{ex}

\begin{ex}
Let $F$ be a free Lie algebra of $2$ generators and $L=F/F^4$.  Then $L$ is a Lie algebra of $2$ generators and class $3$.  
Thus, $\dim L = 5$ and by Hardy's Theorem, $\dim M(L) \leq 7$.  By LB's Theorem, 
	\begin{align*}
		\dim M(L) \leq \sum_{j=1}^3 l_2(j+1) &= l_2(2) + l_2(3) + l_2(4) \\
		&=1 + 2 + \frac{1}{4}[\mu(1)2^4 + \mu(2)2^2 + \mu(4)2] \\
		&=3 + \frac{1}{4}(16-4) \\
		&=3 + 3 =6
	\end{align*}
Thus, we see that in this case, our theorem creates a better upper bound for $\dim M(L)$ than the previously known result.
\end{ex}

\end{document}